\documentclass[preprint,12pt,numbers,sort&compress]{elsarticle}

\usepackage{amssymb}
\usepackage{amsmath,amssymb,amsthm,amsfonts,latexsym,graphicx,subfigure}

\DeclareMathOperator*{\vol}{\ensuremath{Vol}}
\DeclareMathOperator*{\RQ}{\ensuremath{RQ}}
\DeclareMathOperator{\id}{Id}
\newcommand{\R}{\ensuremath{\mathbb{R}}}

\theoremstyle{plain}
\newtheorem{theorem}{Theorem}[section]
\newtheorem{lemma}[theorem]{Lemma}
\newtheorem{proposition}[theorem]{Proposition}
\newtheorem{corollary}[theorem]{Corollary}

\theoremstyle{definition}
\newtheorem{definition}[theorem]{Definition}
\newtheorem{ex}[theorem]{Example}
\theoremstyle{remark}
\newtheorem{remark}[theorem]{Remark}

\allowdisplaybreaks

\journal{Discrete Mathematics}
\date{}

\begin{document}

\begin{frontmatter}



\title{Spectral Theory of Laplace Operators on Oriented Hypergraphs}

\author[inst1]{Raffaella Mulas\footnote{Corresponding author}}

\affiliation[inst1]{organization={Max Planck Institute for Mathematics in the Sciences},
            addressline={}, 
            city={Leipzig},
            postcode={D–04103}, 
            state={},
            country={Germany}}

\author[inst1]{Dong Zhang}

\begin{abstract}
Several new spectral properties of the normalized Laplacian defined for oriented hypergraphs are shown. The eigenvalue $1$ and the case of duplicate vertices are discussed; two Courant nodal domain theorems are established; new quantities that bound the eigenvalues are introduced. In particular, the Cheeger constant is generalized and it is shown that the classical Cheeger bounds can be generalized for some classes of hypergraphs; it is shown that a geometric quantity used to study zonotopes bounds the largest eigenvalue from below, and that the notion of coloring number can be generalized and used for proving a Hoffman-like bound. Finally, the spectrum of the unnormalized Laplacian for Cartesian products of hypergraphs is discussed.
\end{abstract}



\end{frontmatter}

\section{Introduction}
The \emph{oriented hypergraphs} introduced by Shi \cite{Shi92} are hypergraphs with the additional structure that each vertex in a hyperedge is either an input or an output. Such structure allows modeling many real networks, as for instance chemical reaction networks, metabolic networks, neural networks, synchronization networks. The \emph{adjacency} and \emph{unnormalized Laplacian} matrices of oriented hypergraphs were introduced by Reff and Rusnak \cite{ReffRusnak} and the study of their spectral properties has received a lot of attention \cite{orientedhyp2013,orientedhyp2014,orientedhyp2016,orientedhyp2017,orientedhyp2018,orientedhyp2019,orientedhyp2019-2,orientedhyp2019-3}. The \emph{normalized Laplacian} has been established in \cite{Hypergraphs}, and various spectral properties, as well as possible applications, have been studied in \cite{Hypergraphs,Sharp,Master-Stability}. In this work, we bring forward the study of the spectrum for both the normalized and the unnormalized Laplacian in the context of oriented hypergraphs, with a special focus on the first one.

The paper is structured as follows. Section \ref{section:prel} provides an overview of the preliminaries needed in order to discuss the main results. The next five sections focus on the normalized Laplacian. In particular, in Section \ref{section:eigenvalue1} we discuss the eigenvalue $1$ and the case of duplicate vertices; in Section \ref{Section:Courant} we prove two versions of the Courant nodal domain theorem and in Section \ref{section:Cheeger} we discuss the problem of generalizing the Cheeger inequalities to the case of the smallest non-zero eigenvalue of the normalized Laplacian. In Section \ref{section:generalbounds} we prove some general bounds for both the smallest nonzero eigenvalue and the largest eigenvalue, while in Section \ref{section:coloring} we generalize the definition of coloring number and we discuss some of its properties. Finally, in Section \ref{section:unnormalized}, we study the spectrum of the unnormalized Laplacian when considering the Cartesian product of oriented hypergraphs.

\section{Preliminaries}\label{section:prel}
We discuss the preliminaries needed in order to state the main results. In particular, in Section \ref{section:hyp} we present an overview of the basic definitions regarding oriented hypergraphs; in Section \ref{section:operators} we provide an overview of the operators associated to such hypergraphs; in Section \ref{section:minmax} we characterize the eigenvalues of the normalized Laplacian using the min-max principle. Finally, in Section \ref{section:transformations}, we discuss two kinds of hypergraph transformations.
\subsection{Oriented hypergraphs}\label{section:hyp}
\begin{definition}[\cite{ReffRusnak}]
				An \emph{oriented hypergraph} is a pair $\Gamma=(V,H)$ such that $V$ is a finite set of vertices and $H$ is a set such that every element $h$ in $H$ is a pair of disjoint elements $(h_{in},h_{out})$ (input and output) in $\mathcal{P}(V)$. The elements of $H$ are the \emph{oriented hyperedges}. Changing the orientation of a hyperedge $h$ means exchanging its input and output, leading to the pair $(h_{out},h_{in})$. With a little abuse of notation, we shall see $h$ as $h_{in}\cup h_{out}$.
			\end{definition}
			   \begin{definition}
 Given $h\in H$, two vertices $i$ and $j$ are \emph{co-oriented} in $h$ if they belong to the same orientation sets of $h$; they are \emph{anti-oriented} in $h$ if they belong to different orientation sets of $h$.
\end{definition}

\begin{definition}[\cite{Sharp}]
The \emph{degree} of a vertex $i$ is
	\begin{equation*}
	    \deg(i):=\#\text{ hyperedges containing $i$}.
	\end{equation*}The \emph{cardinality} of a hyperedge $h$ is
	\begin{equation*}
	    \# h:=\#(h_{in}\cup h_{out}).	\end{equation*}
\end{definition}

From now on, we fix an oriented hypergraph $\Gamma=(V,H)$ on $n$ vertices $1,\ldots,n$ and $m$ hyperedges $h_1,\ldots, h_m$. We assume that $\Gamma$ has no vertices of degree zero.

 \begin{definition}
    The oriented hypergraph $\Gamma$ is \emph{$d$--regular} if $\deg (i)=d$ for each $i\in V$. $\Gamma$ is \emph{$m$--uniform} if $\#h=m$ for each $h\in H$.
    \end{definition}
    
    \begin{definition}\label{def:c-complete}
    The oriented hypergraph $\Gamma$ is \emph{$c$--complete} for some $c\geq 1$ if, forgetting about the additional structure of inputs and outputs, it has all possible ${n\choose c}$ hyperedges of cardinality $c$. 
    \end{definition}
    \begin{ex}
   Every graph is $2$--uniform. The complete graph is $2$--complete according to Definition \ref{def:c-complete}.
    \end{ex}
    
\begin{definition}[\cite{Hypergraphs}]
The oriented hypergraph $\Gamma$ is \emph{connected} if, for every pair of vertices $v,w\in V$, there exists a path that connects $v$ and $w$, i.e. there exist $v_1,\ldots,v_k\in V$ and $h_1,\ldots,h_{k-1}\in H$ such that $v_1=v$, $v_k=w$, and $\{v_i,v_{i+1}\}\subseteq h_i$ for each $i=1,\ldots,k-1$.
\end{definition}

	\begin{definition}[\cite{Hypergraphs}]
				 The oriented hypergraph $\Gamma=(V,H)$ has $k$ \emph{connected components} if there exist $\Gamma_1=(V_1,H_1),\ldots,\Gamma_k=(V_k,H_k)$ such that:
				\begin{enumerate}
					\item For every $i\in\{1,\ldots,k\}$, $\Gamma_i$ is a connected hypergraph with $V_i\subseteq V$ and $H_i\subseteq H$;
					\item For every $i,j\in\{1,\ldots,k\}$, $i\neq j$, $V_i\cap V_j=\emptyset$ and therefore also $H_i\cap H_j=\emptyset$;
\item $\bigcup V_i=V,\ \  \bigcup H_i =H$.
				\end{enumerate}
			\end{definition}

\subsection{Operators on oriented hypergraphs}\label{section:operators}
\begin{definition}[\cite{Hypergraphs}]
The $n\times m$ \emph{incidence matrix} of $\Gamma$ is $\mathcal{I}:=(\mathcal{I}_{ih})_{i\in V, h\in H}$, where
	\begin{equation*} 
	\mathcal{I}_{ih}:=\begin{cases} 1 & \text{ if }i\in h_{in}\\ -1 & \text{ if }i\in h_{out}\\ 0 & \text{otherwise.} \end{cases}
	\end{equation*}
\end{definition}
 
\begin{definition}[\cite{ReffRusnak}]
       The $n\times n$ diagonal \emph{degree matrix} $D:=D(\Gamma)$ is defined by
    \begin{equation*}
    D_{ij}:=\begin{cases} \deg (i) & \text{if }i=j\\ 0 & \text{otherwise}. \end{cases} 
\end{equation*}
    \end{definition}
 \begin{definition}[\cite{ReffRusnak}]\label{def:adjacencymatrix}
        The $n\times n$ \emph{adjacency matrix} is $A:=A(\Gamma)$, where $A_{ii}:=0$ for each $i=1,\ldots,n$ and
\begin{align*}
        A_{ij}:=& \# \{\text{hyperedges in which }i \text{ and }j\text{ are anti-oriented}\}\\
        &-\# \{\text{hyperedges in which }i \text{ and }j\text{ are co-oriented}\}
\end{align*}for $i\neq j$.
    \end{definition}

\begin{definition}[\cite{Hypergraphs}]\label{def:scalarfg}Let $C(V)$ be the space of functions $f:V\rightarrow\mathbb{R}$, endowed with the scalar product 
\begin{equation*}
    (f,g):=\sum_{i\in V}\deg(i) f(i) g(i).
\end{equation*}The \emph{(normalized) Laplacian} associated to $\Gamma$ is the operator
	\begin{equation*}
	    L:C(V)\rightarrow C(V)
	\end{equation*}such that, given $f:V\rightarrow\mathbb{R}$ and given $i\in V$,
\begin{align*}
Lf(i):=&\frac{\sum_{h: i\text{ input}}\biggl(\sum_{i' \text{ input of }h}f(i')-\sum_{j' \text{ output of }h}f(j')\biggr)}{\deg(i)}\\
&-\frac{\sum_{\hat{h}: i\text{ output}}\biggl(\sum_{\hat{i} \text{ input of }\hat{h}}f(\hat{i})-\sum_{\hat{j} \text{ output of }\hat{h}}f(\hat{j})\biggr)}{\deg(i)}.
\end{align*}
	\end{definition}
	\begin{remark}
Note that, as well as the graph normalized Laplacian, $L$ can be rewritten in a matrix form as
\begin{equation*}
    L=\id-D^{-1}A,
\end{equation*}where $\id$ is the $n\times n$ identity matrix. To see this, observe that, given $f:V\rightarrow\mathbb{R}$ and $i\in V$,
\begin{align*}
    Lf(i)=&\frac{\sum_{h: i\text{ input}}\biggl(\sum_{i' \text{ input of }h}f(i')-\sum_{j' \text{ output of }h}f(j')\biggr)}{\deg(i)}\\
&-\frac{\sum_{\hat{h}: i\text{ output}}\biggl(\sum_{\hat{i} \text{ input of }\hat{h}}f(\hat{i})-\sum_{\hat{j} \text{ output of }\hat{h}}f(\hat{j})\biggr)}{\deg(i)}\\
=&\frac{\deg(i)f(i)-\sum_{j\neq i}A_{ij}f(j)}{\deg(i)}\\
=&f(i)-\frac{1}{\deg(i)}\sum_{j\neq i}A_{ij}f(j).
\end{align*}
	\end{remark}
$L$ is not necessarily a symmetric matrix, but it is a symmetric operator with respect to the scalar product that we use. Also, if we generalize the symmetric normalized Laplacian introduced by Chung \cite{Chung} as
\begin{equation*}
    \mathcal{L}:=\id-D^{-1/2}AD^{-1/2},
\end{equation*}it is easy to see that $L=D^{-1/2}\mathcal{L}D^{1/2}$, therefore $L$ and $\mathcal{L}$ are similar. In particular, they have the same eigenvalues counted with multiplicity. This allows us to apply the theory of symmetric matrices in order to study the eigenvalues of $L$.
\begin{remark}\label{remark:trace}
It is easy to see that the trace of $L$ is equal to $n$. Therefore, also the sum of its eigenvalues is equal to $n$.
\end{remark}	
	
\begin{definition}[\cite{ReffRusnak}]\label{def:unnormalized-Laplacian}
The \emph{unnormalized Laplacian} associated to $\Gamma$ is the operator
	\begin{equation*}
	    \Delta:C(V)\rightarrow C(V)
	\end{equation*}such that, given $f:V\rightarrow\mathbb{R}$ and given $i\in V$,
\begin{align*}
\Delta f(i):=&\sum_{h\in H\,:\, h_{in}\ni i}\left(\sum_{j'\in h_{in}}f(j')-\sum_{j\in h_{out}}f(j)\right)\\
&-\sum_{h\in H\,:\, h_{out}\ni i}\left(\sum_{j'\in h_{in}}f(j')-\sum_{j\in h_{out}}f(j)\right).
\end{align*}

\end{definition}
\begin{remark}The unnormalized Laplacian $\Delta$ can be written in a matrix form as $\Delta=D-A$.

Also, the Laplacian $L$ is such that \begin{equation*}Lf(i)=\frac{1}{\deg(i)}\cdot\Delta f(i) \quad\text{for all }i\in V.\end{equation*}Therefore, it is easy to see that, if $\Gamma$ is $d$--regular,
\begin{align*}
    \lambda \text{ is an eigenvalue for }\Delta &\iff \frac{\lambda}{d} \text{ is an eigenvalue for }L\\
    &\iff d-\lambda \text{ is an eigenvalue for }A.
\end{align*}
\end{remark}
\begin{definition}[\cite{Hypergraphs}]\label{def:scalarproductgamma}
Let $C(H)$ be the space of functions $\gamma:H\rightarrow\mathbb{R}$, endowed with the scalar product 
\begin{equation*}
    (\gamma,\tau)_H:=\sum_{h\in H}\gamma(h)\tau(h).
\end{equation*}The \emph{hyperedge-Laplacian} associated to $\Gamma$ is the operator
	\begin{equation*}
	    L^H:C(H)\rightarrow C(H)
	\end{equation*}such that, given $\gamma:H\rightarrow\mathbb{R}$ and given $h\in H$,
\begin{align*}
L^H\gamma(h):=&\sum_{i \text{ input of }h}\frac{\sum_{h': i\text{ input}}\gamma(h')-\sum_{h'': i\text{ output}}\gamma(h'')}{\deg (i)}\\
&-\sum_{j \text{ output of }h}\frac{\sum_{\hat{h}': j\text{ input}}\gamma(\hat{h}')-\sum_{\hat{h}'': j\text{ output}}\gamma(\hat{h}'')}{\deg (j)}.
\end{align*}
\end{definition}
\subsection{Min-max principle}\label{section:minmax}
We recall that $L$ has $n$ real, non-negative eigenvalues that we denote by
\begin{equation*}
    \lambda_1\leq \ldots\leq\lambda_n.
\end{equation*}Analogously, $L^H$ has $m$ real, non-negative eigenvalues,
\begin{equation*}
    \mu_1\leq \ldots\leq\mu_m.
\end{equation*}As shown in \cite{Hypergraphs}, the non-zero spectrum of $L$ and $L^H$ coincides. Also, the multiplicity of the eigenvalue $0$ for $L$, denoted $m_V$, and the multiplicity of $0$ for $L^H$, denoted $m_H$, are such that
\begin{equation*}
    m_V-m_H=n-m.
\end{equation*}
By the Courant-Fischer-Weyl min-max principle, we can characterize all eigenvalues of $L$ and $L^H$ as follows. Given a function $f\in C(V)$, its \emph{Rayleigh Quotient} is
\begin{equation*}
    \RQ(f)=\frac{\sum_{h\in H}\left(\sum_{i\text{ input of }h}f(i)-\sum_{j\text{ output of }h}f(j)\right)^2}{\sum_{i\in V}\deg(i)f(i)^2}
\end{equation*}and similarly, given $\gamma\in C(H)$,
\begin{equation*}
    \RQ(\gamma)=\frac{\sum_{i\in V}\frac{1}{\deg (i)}\cdot \biggl(\sum_{h': i\text{ input}}\gamma(h')-\sum_{h'': i\text{ output}}\gamma(h'')\biggr)^2}{\sum_{h\in H}\gamma(h)^2}.
\end{equation*}By the min-max principle, for $k=1,\ldots,n$,
\begin{align*}
    \lambda_k&=\min_{\substack{f\in C(V),\,(f,f_j)=0\\j=1,\ldots,k-1}}\RQ(f)\\
    &=\max_{\substack{f\in C(V),\,(f,f_l)=0\\l=k,\ldots,n}}\RQ(f),
\end{align*}where each $f_j$ is an eigenfunction of $\lambda_j$. Also, the functions $f_k$ realizing such a minimum or maximum are the corresponding eigenfunctions of $\lambda_k$. Analogously, for $k=1,\ldots,m$,
\begin{align*}
    \mu_k&=\min_{\substack{\gamma\in C(H),\,(\gamma,\gamma_j)_H=0\\j=1,\ldots,k-1}}\RQ(\gamma)\\
    &=\max_{\substack{\gamma\in C(H),\,(\gamma,\gamma_l)_H=0\\l=k,\ldots,m}}\RQ(\gamma),
\end{align*}where each $\gamma_j$ is an eigenfunction of $\mu_j$, and the functions $\gamma_k$ realizing such a minimum or maximum are the corresponding eigenfunctions of $\mu_k$.
In particular,
\begin{equation*}
    \lambda_1=\min_{f\in C(V)}\RQ(f),\quad \mu_1=\min_{\gamma\in C(H)}\RQ(\gamma)
    \end{equation*}
    and
    \begin{equation*}
    \lambda_n=\max_{f\in C(V)}\RQ(f)=\max_{\gamma\in C(H)}\RQ(\gamma).
\end{equation*}
\subsection{Hypergraph transformations}\label{section:transformations}
\begin{definition}[\cite{orientedhyp2014}]
Given $\hat{v}\in V$, we let $\Gamma- \hat{v}:=(\hat{V},\hat{H})$, where:
\begin{itemize}
    \item $\hat{V}=V\setminus\{\hat{v}\}$, and
    \item $\hat{H}=\{h\setminus\{\hat{v}\}:h\in H\}$.
\end{itemize}$\Gamma- \hat{v}$ is obtained from $\Gamma$ by a \emph{weak vertex deletion} of $\hat{v}$. $\Gamma$ is obtained from $\Gamma- \hat{v}$ by a \emph{weak vertex addition} of $\hat{v}$. We also allow empty hyperedges.
\end{definition}
\begin{lemma}\label{lemma:Cauchy}If $\hat{\Gamma}$ is obtained from $\Gamma$ by weak-deleting $r$ vertices,
\begin{equation*}
        \lambda_{k}(\Gamma)\leq \lambda_k(\hat{\Gamma})\leq \lambda_{k+r}(\Gamma)\quad\text{for all }k\in\{1,\ldots,n-r\}.
    \end{equation*}
\end{lemma}
\begin{proof}
By the Cauchy Interlacing Theorem \cite[Theorem 4.3.17]{MatrixAnalysis},
\begin{equation*}
    \lambda_{k}(\Gamma)\leq \lambda_k(\Gamma- \hat{v})\leq \lambda_{k+1}(\Gamma)\quad\text{for all }k\in\{1,\ldots,n-1\}.
    \end{equation*}By induction, one proves the claim.
\end{proof}
\begin{definition}[\cite{ReffRusnak}]
Let $\Gamma=(V,H)$ be an oriented hypergraph with $V=\{v_1,\ldots,v_n\}$ and $H=\{h_1,\ldots,h_m\}$. We construct the \emph{dual hypergraph} of $\Gamma$ as $\Gamma^{\top}:=(V',H')$, where $V'=\{v'_1,\ldots,v'_m\}$, $H'=\{h'_1,\ldots,h'_n\}$ and
    \begin{equation*}
 v'_j\in h'_i \text{ as input (resp. output)}\iff v_i\in h_j \text{ as input (resp. output)}.
    \end{equation*}
 \begin{remark}
   Clearly, $\mathcal{I}(\Gamma^{\top})=\mathcal{I}(\Gamma)^\top$, $m_V(\Gamma)=m_H(\Gamma^{\top})$ and $m_H(\Gamma)=m_V(\Gamma^{\top})$. For some hypergraphs, as we shall see in Lemma \ref{lemma:dm} below, also the non-zero eigenvalues of $\Gamma$ and $\Gamma^{\top}$ are related.
 
 \end{remark}
\end{definition}

    \begin{lemma}\label{lemma:dm}Let $\Gamma$ be $d$--regular and $m$--uniform. Then, 
    \begin{equation*}
        \lambda \text{ is an eigenvalue for } \Gamma \,\iff\, \frac{d}{m}\lambda \text{ is an eigenvalue for } \Gamma^{\top}.
    \end{equation*}
    \end{lemma}
  
    \begin{proof}
    The eigenvalues of $\Gamma$ are the min-max of the Rayleigh Quotient 
    \begin{equation*}
      \RQ(f:V\rightarrow\mathbb{R})=\frac{1}{d} \cdot \frac{\sum_{h\in H}\biggl(\sum_{v_{\text{in}}\in h \text{ input}}f(v_{\text{in}})-\sum_{v_{\text{out}}\in h \text{ output}}f(v_{\text{out}})\biggr)^2}{\sum_{i\in V} f(i)^2}.
    \end{equation*}On the other hand, the eigenvalues of $\Gamma^{\top}$ are the min-max of 
    \begin{align*}
      \RQ(\gamma:H\rightarrow\mathbb{R})&=\frac{1}{m} \cdot \frac{\sum_{v'\in V'}\biggl(\sum_{h'_{\text{in}}: v\text{ input}}\gamma(h'_{\text{in}})-\sum_{h'_{\text{out}}: v'\text{ output}}\gamma(h'_{\text{out}})\biggr)^2}{\sum_{h'\in H'}\gamma(h')^2}\\
      &=\frac{d}{m} \cdot \RQ(f:V\rightarrow\mathbb{R}).
    \end{align*}
    \end{proof}
      \begin{remark}In the case of graphs, Lemma \ref{lemma:dm} is trivial because every graph is $2$--uniform, therefore if a $d$--regular graph has a dual graph it must have $d=m=2$.
    \end{remark}
\section{Eigenvalue 1 and duplicate vertices}\label{section:eigenvalue1}
In the case of graphs, it is well-known that \emph{duplicate vertices}, that is, vertices that share the same neighbors, produce the eigenvalue $\lambda=1$ \cite{duplicationgraphs}. We show that this is true also for the more general case of hypergraphs.
\begin{lemma}
$1$ is an eigenvalue for $L$ with eigenfunction $f$ if and only if $0$ is an eigenvalue for $A$ with eigenfunction $f$. In particular, the multiplicity of $1$ for $L$ equals the multiplicity of $0$ for $A$.
\end{lemma}
\begin{proof}Observe that
\begin{equation*}
    Lf=f\iff (\id-D^{-1}A)f=f\iff \id f-D^{-1}Af=f\iff Af=0.
\end{equation*}This proves the claim.
\end{proof}
\begin{definition}\label{def:duplicate}
Two vertices $i$ and $j$ are \emph{duplicate vertices} if the corresponding rows/columns of the adjacency matrix are the same, that is, $A_i=A_j$ and therefore
\begin{equation*}
    A_{il}=A_{jl}\quad\text{for each }l\in V.
\end{equation*}In particular, $A_{ij}=A_{jj}=0$.
\end{definition}
\begin{remark}
In the case of graphs, Definition \ref{def:duplicate} coincides with the usual definition of duplicate vertices.
\end{remark}
\begin{lemma}\label{lemma:duplicate}
If $i$ and $j$ are duplicate vertices, let $f:V\rightarrow\mathbb{R}$ be such that $f(i)=-f(j)\neq 0$ and $f=0$ otherwise. Then, $Lf=f$, that is, $1$ is an eigenvalue and $f$ is a corresponding eigenfunction.
\end{lemma}
\begin{proof}
It is easy to see that, by definition of $f$,
\begin{itemize}
    \item $Lf(i)=f(i)$,
    \item $Lf(j)=f(j)$, and
    \item For each $l\neq i,j$,
    \begin{equation*}
        Lf(l)=-\frac{1}{\deg (i)}\bigl(A_{li}f(i)+A_{lj}f(j)\bigr)=0=f(l).
    \end{equation*}
\end{itemize}
\end{proof}
\begin{corollary}
If there are $\hat{n}$ duplicate vertices, then the multiplicity of $1$ is at least $\hat{n}-1$.
\end{corollary}
\begin{proof}
Assume, up to reordering, that $1,\ldots, \hat{n}$ are duplicate vertices. For each $i=1,\ldots,\hat{n}-1$, let $f_i:V\rightarrow\mathbb{R}$ be such that $f_i(i)=1$, $f_i(i+1)=-1$ and $f_i=0$ otherwise. Then, by Lemma \ref{lemma:duplicate} the $f_i$'s are eigenfunctions corresponding to the eigenvalue $1$. Also, $\dim(\mathrm{span}(f_1,\ldots,f_{\hat{n}-1}))=\hat{n}-1$, therefore the multiplicity of $1$ is at least $\hat{n}-1$.
\end{proof}
Also, \cite[Lemma 10]{Cheeger-like-graphs} for duplicate vertices can be generalized as follows.
\begin{lemma}\label{lemduplicate}
If $i$ and $j$ are duplicate vertices and $f$ is an eigenfunction for an eigenvalue $\lambda\neq 1$ of $L$, then $$f(i)=\frac{\deg (j)}{\deg (i)}f(j).$$
\end{lemma}
\begin{proof}An eigenvalue $\lambda$  of $L$ with eigenfunction $f$ satisfies, for each vertex $l$,
\begin{equation*}
    \lambda f(l)=L f(l)=f(l)-\frac{1}{\deg (l)}\sum_{k\neq l}A_{lk}f(k),
\end{equation*}that is,
\begin{equation*}
\frac{1}{\deg (l)}\sum_{k\neq l}A_{lk}f(k)=f(l)(1-\lambda).
\end{equation*}
Therefore, since this is true, in particular, for $i$ and $j$ and by assumption these are duplicate vertices,
\begin{equation*}
    \frac{1}{\deg (i)}\sum_{k\neq i}A_{ik}f(k)=f(i)(1-\lambda)=\frac{\deg (j)}{\deg (i)}f(j) (1-\lambda).
\end{equation*}Since by assumption $\lambda\neq 1$, this implies that $$f(i)=\frac{\deg (j)}{\deg (i)}f(j).$$
\end{proof}

\section{Courant nodal domain theorems}\label{Section:Courant}

We establish two Courant nodal domain theorems for oriented hypergraphs. In particular, in Section \ref{signless-nodal} we prove a signless nodal theorem that holds for all oriented hypergraphs; in Section \ref{positive-nodal} we define positive and negative domains and we establish the corresponding Courant nodal domain theorem for hypergraphs that have only inputs. We refer the reader to \cite{nodalgraphs} for nodal domain theorems on graphs.
\subsection{Signless nodal domain theorem}\label{signless-nodal}

\begin{definition}\label{def:nodal-domain}
Given a function $f:V\to \R$, we let $\mathrm{supp}(f):=\{i\in V: f(i)\ne0\}$ be the support set of $f$. A \emph{nodal domain} of $f$ is a connected component of the hypergraph that has vertex set $V$ and hyperedge set
\begin{equation*}
    H\cap \mathrm{supp}(f):=\{h\cap \mathrm{supp}(f): h\in H\}.
\end{equation*}
\end{definition}

\begin{theorem}\label{thm:Courant-nodal}
If $f$ is an eigenfunction of the $k$-th eigenvalue $\lambda_k$ and this has multiplicity $r$, then the number of nodal domains of $f$ is smaller than or equal to $k+r-1$.
\end{theorem}
\begin{remark}\label{remark:nodal-domain}
For graphs, the above definition of nodal domain does not coincide with the classical one. The reason why we made this choice is that, if we generalize the classical definitions using the positive and negative nodal domains, then Theorem \ref{thm:Courant-nodal} cannot hold. In fact, the usual nodal domain in graph theory is a connected component of the graph that has edge set $H\cap \mathrm{supp}_+(f)$ or $H\cap \mathrm{supp}_-(f)$, where $$\mathrm{supp}_\pm(f):=\{i\in V:\pm f(i)>0\}.$$ But using this definition for hypergraphs, the  number of nodal domains might be too large to satisfy Theorem \ref{thm:Courant-nodal}. A counterexample is shown below.

Let $\Gamma:=(V,H)$, where $V:=\{1,\ldots,8\}$, and
\begin{align*}
    H:=&\{\{1,2,3,4\},\{3,4,5,6\},\{5,6,7,8\},\{7,8,1,2\},\{1,3\},\{1,4\},\{3,5\},\\
    &\{3,6\},\{5,7\},\{5,8\}\},
\end{align*}

with the assumption that all vertices are inputs for all hyperedges in which they are contained. Then, one can check that $m_V=1$. However, the corresponding eigenfunction $f$ of $\lambda_1=0$ defined by $f(1)=f(2)=f(5)=f(6)=1$ and $f(3)=f(4)=f(7)=f(8)=-1$ has $2$ positive nodal domains and $2$ negative nodal domains. Thus, the total number of nodal domains of $f$ is $4$, which is larger than $k+r-1=1+1-1=1$. Definition \ref{def:nodal-domain} can overcome this problem. 

\end{remark}
\begin{proof}[Proof of Theorem \ref{thm:Courant-nodal}]
Suppose the contrary, that is, $f$ is an eigenfunction of $\lambda_k$ with multiplicity $r$, and $f$ has at least $k+r$ nodal domains whose vertex sets are denoted by $V_1,\ldots,V_{k+r}$. For simplicity, we assume that
\begin{equation*}
   \lambda_k=\lambda_{k+1}=\ldots=\lambda_{k+r-1}<\lambda_{k+r}.
\end{equation*}
Consider a linear function-space $X$ spanned by $f|_{V_1},\ldots,f|_{V_{k+r}}$, where the restriction $f|_{V_i}$ is defined by
\begin{equation*}
    f|_{V_i}(j)=\begin{cases}f(j),&\text{ if }j\in V_i,\\ 0,&\text{ if } j\not\in V_i.\end{cases}
\end{equation*}
Since $V_1,\ldots,V_{k+r}$ are pairwise disjoint, $\dim X=k+r$. Given $g\in X\setminus 0$, there exists $(t_1,\ldots,t_{k+r})\ne\vec0$ such that
\begin{equation*}
    g=\sum_{i=1}^{k+r} t_i f|_{V_i}.
\end{equation*}
Hence,
\begin{align*}
&\sum_{h\in H}\left(\sum_{i\text{ input of }h}g(i)-\sum_{j\text{ output of }h}g(j)\right)^2
\\=& \sum_{h\in H}\left(\sum_{i\in h_{in}}\sum_{l=1}^{k+r}t_lf|_{V_l}(i)-\sum_{j\in h_{out}}\sum_{l=1}^{k+r}t_lf|_{V_l}(j)\right)^2
\\=&\sum_{h\in H}\left(\sum_{l=1}^{k+r}t_l\left(\sum_{i\in h_{in}}f|_{V_l}(i)-\sum_{j\in h_{out}}f|_{V_l}(j)\right)\right)^2
\\=&\sum_{h\in H}\sum_{l=1}^{k+r}t_l^2\left(\sum_{i\in h_{in}}f|_{V_l}(i)-\sum_{j\in h_{out}}f|_{V_l}(j)\right)^2 \\
&\;\;\;(\because \forall h,\, \sum_{i\in h_{in}}f|_{V_l}(i)-\sum_{j\in h_{out}}f|_{V_l}(j)\ne 0 \text{ for at most one }l) 
\\=&\sum_{l=1}^{k+r}t_l^2 \sum_{h\in H}\left(\sum_{i\in h_{in}}f|_{V_l}(i)-\sum_{j\in h_{out}}f|_{V_l}(j)\right)^2
\\=&\sum_{l=1}^{k+r}t_l^2 \sum_{h\in H:\, h\cap V_{l}\ne\emptyset}\left(\sum_{i\in h_{in}}f|_{V_l}(i)-\sum_{j\in h_{out}}f|_{V_l}(j)\right)^2
\\=&\sum_{l=1}^{k+r}t_l^2 \sum_{h\in H:\, h\cap V_{l}\ne\emptyset}\left(\sum_{i\in h_{in}}f(i)-\sum_{j\in h_{out}}f(j)\right)^2
\\=&\sum_{l=1}^{k+r}t_l^2\sum_{i\in V_l}f(i)\left(\sum_{\substack{h\in H: \\ h_{in}\ni i}}\left(\sum_{j'\in h_{in}}f(j')-\sum_{j\in h_{out}}f(j)\right)-\sum_{\substack{h\in H:\\ h_{out}\ni i}}\left(\sum_{j'\in h_{in}}f(j')-\sum_{j\in h_{out}}f(j)\right)\right)
\\=&\sum_{l=1}^{k+r}t_l^2\sum_{i\in V_l}f(i) \lambda_k\deg(i)f(i)
\\=&\lambda_k\sum_{l=1}^{k+r}t_l^2\sum_{i\in V_l}\deg(i)f(i)^2
\\=&\lambda_k\sum_{l=1}^{k+r}\sum_{i\in V}\deg(i)\Biggl(t_l f|_{V_l}(i)\Biggr)^2
\\=&\lambda_k\sum_{i\in V}\deg(i)\sum_{l=1}^{k+r}\Biggl(t_l f|_{V_l}(i)\Biggr)^2
\\=& \lambda_k\sum_{i\in V}\deg(i)\Biggl(\sum_{l=1}^{k+r}t_l f|_{V_l}(i)\Biggr)^2\;\;\; (\because V_1,\ldots,V_{k+r}\text{ are pairwise disjoint})
\\=&\lambda_k\sum_{i\in V}\deg(i)g(i)^2.
\end{align*}

Therefore, $\RQ(g)=\lambda_k$. 
Now, let $\mathcal{X}_{k+r}$ be the family of all $(k+r)$--dimensional subspaces of $C(V)$. By the min-max principle,
\begin{align*}
\lambda_{k+r}&=\min\limits_{ X'\in\mathcal{X}_{k+r}}\max\limits_{g'\in X'\setminus0} \frac{\sum_{h\in H}\left(\sum_{i\text{ input of }h}g'(i)-\sum_{j\text{ output of }h}g'(j)\right)^2}{\sum_{i\in V}\deg(j)g'(i)^2}
\\&\le \max\limits_{g'\in X\setminus0} \frac{\sum_{h\in H}\left(\sum_{i\text{ input of }h}g'(i)-\sum_{j\text{ output of }h}g'(j)\right)^2}{\sum_{i\in V}\deg(j)g'(i)^2}
\\&=\lambda_k,
\end{align*}
which leads to a contradiction.
\end{proof}

\subsection{Positive and negative nodal domain theorem}\label{positive-nodal}

\begin{definition}
Given a function $f:V\to \R$, a \emph{positive nodal domain} of $f$ is a connected component of the hypergraph that has vertex set $V$ and hyperedge set
\begin{equation*}
    H\cap \mathrm{supp}_+(f):=\{h \cap \mathrm{supp}_+(f): h\in H\},
\end{equation*}
where the notion $\mathrm{supp}_+(f)$ is already used in Remark \ref{remark:nodal-domain}. Analogously, a \emph{negative nodal domain} of $f$ is a connected component of the hypergraph that has vertex set $V$ and hyperedge set $H\cap \mathrm{supp}_-(f)$.
\end{definition}
\begin{theorem}\label{thm:Courant-nodal-reverse}
Let $\Gamma=(V,H)$ be an oriented hypergraph with only inputs. If $f$ is an eigenfunction of the $k$-th eigenvalue $\lambda_k$ and this has multiplicity $r$, then the number of positive and negative nodal domains of $f$ is smaller than or equal to $n-k+r$.
\end{theorem}

\begin{proof}
Since $\Gamma$ has only inputs, $A_{ij}<0$ whenever there is a hyperedge $h\in H$ that contains both $i$ and $j$, and $A_{ij}=0$ otherwise.  Thus, the connectivity 
in the hypergraph $\Gamma$ is equivalent to the  connectivity in the weighted graph that has adjacency matrix $-A$. Now, the symmetric matrix $-\mathcal{L}= D^{-\frac12}AD^{-\frac12} -I$, which is isospectral to $-L$, is a Schr\"odinger operator on this weighted graph in the sense of
\cite{nodalgraphs}, and thus the nodal domain theorem for graphs in \cite{nodalgraphs} can be applied to derive that the number of nodal domains of an eigenfunction of   $\lambda_k(L)=-\lambda_{n-k+1}(-\mathcal{L})$ does not exceed $(n-k+1)+r-1=n-k+r$.
\end{proof}

\begin{remark}
The upper bound $n-k+r$ in Theorem \ref{thm:Courant-nodal-reverse} is non-increasing with respect to the eigenvalues, while to the best of our knowledge, all existing Courant nodal domain theorems in literature regarding positive and negative domains have non-decreasing upper bounds with respect to the eigenvalues (i.e., larger eigenvalues usually have more nodal domains). The reason might be that, in the graph case, the Courant nodal domain theorem of the signless Laplacian can be derived by that of the Laplacian directly, and the signless Laplacian has no geometric and PDE analogs, thus no author investigated such a peculiar ``reversed version''. To some extent, Theorem  \ref{thm:Courant-nodal-reverse} can be seen as the first Courant nodal domain theorem for the signless Laplacian on hypergraphs. 
\end{remark}

\begin{remark}
In the first version of this manuscript, the proof of Theorem \ref{thm:Courant-nodal-reverse} was longer and more complicated. One of the anonymous referees noticed that the results on Schr\"odinger operators for weighted graphs in \cite{nodalgraphs} could have been applied, and this allowed to prove Theorem \ref{thm:Courant-nodal-reverse} in a short and elegant way. It is worth noting that the results in \cite{nodalgraphs} can also be used in order to recover the nodal domain theorem in its usual form, since for graphs $A_{ij}=1$ if there is an edge between $i$ and $j$, and $A_{ij}=0$ otherwise.
\end{remark}

\begin{remark}
The Courant nodal domain theorems \ref{thm:Courant-nodal} and \ref{thm:Courant-nodal-reverse} also hold  for the unnormalized Laplacian. In order to see it, it is enough to remove the vertex degrees in the proofs of theorems \ref{thm:Courant-nodal} and \ref{thm:Courant-nodal-reverse}.  
\end{remark}

\section{Generalized Cheeger problem}\label{section:Cheeger}
We propose a generalization of the classical Cheeger constant and we prove that, for some classes of hypergraphs, the Cheeger inequalities involving the smallest non-zero eigenvalue of $L$, that we denote by $\lambda_{\min}$, can be generalized.

Recall that, for a connected graph $G$, as shown in \cite{Chung},
\begin{enumerate}
    \item $\lambda_{\min}=\lambda_2$ and the \emph{harmonic functions}, i.e. the eigenfunctions of $0$, are exactly the constant functions.
    \item The Cheeger constant is 
    \begin{equation*}
	h:=\min_{\substack{\emptyset\neq S\subsetneq V,\\ \vol(S)\leq\frac{\vol(V)}{2}}}\frac{\# E(S,\bar{S})}{\vol(S)}
	\end{equation*}where, given $\emptyset\neq S\subsetneq V$, $\bar{S}:=V\setminus S$, $\# E(S,\bar{S})$ denotes the number of edges with one endpoint in $S$ and the other in $\bar{S}$, and $\vol(S):=\sum_{i\in S}\deg(i)$. The Cheeger inequalities hold:
\begin{equation}\label{eq:Cheegergraphs}
    \frac{1}{2}h^2\leq \lambda_2\leq 2h.
\end{equation}In particular, \eqref{eq:Cheegergraphs} is proved using the fact that, by the min-max principle, knowing that the harmonic functions are exactly the constants, one can write
\begin{equation*}
    \lambda_2=\min_{\substack{f\in C(V),\\\sum_{i\in V}\deg(i)f(i)=0}}\RQ(f).
\end{equation*}Also, in this case, the orthogonality to the constants allows us to say that an eigenfunction $f$ for $\lambda_2$ has to achieve both positive and negative values and therefore we can partition the vertex set as
    \begin{equation*}
        V=\{i:f(i)\geq 0\}\sqcup \{j:f(j)< 0\},
    \end{equation*}and the proof of \eqref{eq:Cheegergraphs} is also based on this.
\end{enumerate}For an oriented hypergraph $\Gamma$, things change because:
\begin{enumerate}
    \item While for graphs we know that $m_V$ equals the number of connected components of the graph, this is no longer true for hypergraphs. In particular, a connected hypergraph might have $m_V=0$ and, on the other hand, a hypergraph with one single connected component might have $m_V>1$, as shown in \cite{Hypergraphs}. Therefore, even if we assume connectivity, we cannot infer that $\lambda_{\min}=\lambda_2$.
    \item The constants are eigenfunctions for $0$ if and only if, for each hyperedge $h$,
    \begin{equation*}
        \# h_{in}=\# h_{out},
    \end{equation*}as shown in \cite{Hypergraphs}. Therefore, in general, we cannot use the orthogonality to the constants, Furthermore, if we assume this condition and we restrict to a smaller class of hypergraphs, we can state that
    \begin{equation*}
        \lambda_2=\min_{f\,:\,\sum_{i\in V}f(i)\deg (i)=0}\RQ(f),
    \end{equation*}but we still cannot infer that $\lambda_2=\lambda_{\min}$. If $m_V>1$, we need to consider also the orthogonality to the other eigenfunctions of $0$, and these eigenfunctions are not known a priori.
\end{enumerate}Therefore, the problem of generalizing \eqref{eq:Cheegergraphs} to the case of oriented hypergraphs is very challenging. Here we generalize the Cheeger constant and we prove that, for some classes of hypergraphs, either the lower bound or the upper bound in \eqref{eq:Cheegergraphs} can be generalized.
 \begin{definition}
Given $\emptyset\neq S\subseteq V$, we let
		\begin{equation*}
		    \vol(S):=\sum_{i\in S}\deg (i),
		\end{equation*}we let
		\begin{equation*}
		\tilde{e}(S):=\sum_{h\in H}\biggl(\#\text{inputs of $h$ in $S$}-\#\text{outputs of $h$ in $S$}\biggr)^2\end{equation*}
	 and we let
		\begin{equation*}
		    \tilde{\nu}(S):=\frac{\tilde{e}(S)}{\vol(S)}.
		\end{equation*}We define a generalization of the Cheeger constant as
		\begin{equation*}
		\tilde{h}:=\min_{\substack{\emptyset\neq S\subseteq V\,:
		\\\vol S\leq \frac{1}{2}\vol V}}\tilde{\nu}(S).
		\end{equation*}
    \end{definition}   \begin{remark}
    In the case of graphs, each edge $e$ has exactly one input and exactly one output, therefore
 \begin{equation*}
	\biggl(\#\text{inputs of $e$ in $S$}-\#\text{outputs of $e$ in $S$}\biggr)\in\{0,1\},
	\end{equation*}which implies that
	\begin{align*}
		\tilde{e}(S)&=\sum_{e\in E}\biggl(\#\text{inputs of $h$ in $S$}-\#\text{outputs of $h$ in $S$}\biggr)^2\\
		&=\sum_{e\in E}\biggl(\#\text{inputs of $h$ in $S$}-\#\text{outputs of $h$ in $S$}\biggr)\\
		&=\sum_{\substack{e\in E\text{ with exactly}\\ \text{one endpoint in $S$}}}1\\
		&=\#\{e\in E\text{ with exactly one endpoint in $S$}\}\\
		&=\# E(S,\bar{S}).
		\end{align*} Hence, $\tilde{h}$ coincides with the classical Cheeger constant when $\Gamma$ is a graph. The geometrical meaning of $\tilde{h}$, in particular, is the same as $h$: we want to divide the vertex set into two disjoint sets that are as big as possible (in terms of the volume) and so that there is as little flow as possible from one to the other.
    \end{remark}
   \begin{remark}If $\#V\geq 2$, let $i$ be a vertex of minimum degree. Then $\deg (i)\leq \vol V/2$ and 
   \begin{equation*}
       \tilde{\nu}(S)=\frac{\deg (i)}{\deg (i)}=1.
   \end{equation*}Therefore, for $\#V\geq 2$, the generalized Cheeger constant is well-defined and $\tilde{h}\leq 1$.
   \end{remark}
    \begin{remark}
	We have that
		\begin{align*}
		\tilde{h}=0\iff &\exists S\subsetneq V, S\neq \emptyset : \vol S \leq \frac{\vol V}{2} \text{ and }\\ &\forall h,\, \#\text{inputs of $h$ in $S$}=\#\text{outputs of $h$ in $S$}.
		\end{align*}\end{remark}
    \begin{remark}Given $S\subset V$, let $f_S:V\rightarrow\mathbb{R}$ be $1$ on $S$ and $0$ on $\bar{S}$. Then, the Rayleigh Quotient of $f_S$ is given by
    \begin{align*}
\RQ(f_S)&=\frac{\sum_{h\in H}\biggl(\#\text{inputs of $h$ in $S$}-\#\text{outputs of $h$ in $S$}\biggr)^2}{\sum_{i\in S}\deg (i)}\\
&=\tilde{\nu}(S).
    \end{align*}In particular,
    \begin{equation}\label{eq:f_S}
        \lambda_n\geq \RQ(f_S)=\tilde{\nu}(S).
    \end{equation}
    
    \end{remark}

 \subsection{Cheeger upper bounds}
    \begin{lemma}\label{lemma:upper}If $m_V=1$ and $\# h_{in}=\# h_{out}$ for each $h$,
    \begin{equation*}
        \lambda_{\min}\leq 2\tilde{h}.
    \end{equation*}
    \end{lemma}
    \begin{proof}We generalize the proof of the upper Cheeger-bound in \cite{JJTDA}. Given $S\subseteq V$, let $f:V\rightarrow\mathbb{R}$ be such that $f:=1$ on $S$ and $f:=-\alpha$ on $\bar{S}$, where $\alpha$ is such that $\sum_{i\in V} \deg(i)f(i)=0$, i.e.
	\begin{equation*}
	\alpha=\frac{\sum_{i\in S}\deg (i)}{\sum_{j\in \bar{S}}\deg (j) }=\frac{\vol S}{\vol \bar{S}}.
	\end{equation*}We also assume that $\vol S\leq \vol \bar{S}$, so that $\alpha\leq 1$. Since we are assuming that $m_V=1$, $\lambda_{\min}=\lambda_2$. Also, since we are assuming that $\# h_{in}=\# h_{out}$ for each $h$, the constants are the harmonic functions. By construction, $f$ is orthogonal to the constants, therefore
	\begin{align*}
	\lambda_{\min}&\leq \RQ(f)\\
	&=\frac{(1+\alpha)^2\cdot\biggl(\sum_{h\in H}\biggl(\#\text{inputs of $h$ in $S$}-\#\text{outputs of $h$ in $S$}\biggr)^2\biggr)}{\sum_{i\in S}\deg (i)+\sum_{j\in \bar{S}}\deg (j)\cdot\alpha^2}\\
	&=\frac{(1+\alpha)^2\cdot\biggl(\sum_{h\in H}\biggl(\#\text{inputs of $h$ in $S$}-\#\text{outputs of $h$ in $S$}\biggr)^2\biggr)}{\sum_{i\in S}\deg (i)+\sum_{j\in \bar{S}}\deg (j)\cdot\alpha^2}\\
		&=\frac{(1+\alpha)\cdot\biggl(\sum_{h\in H}\biggl(\#\text{inputs of $h$ in $S$}-\#\text{outputs of $h$ in $S$}\biggr)^2\biggr)}{\vol S}\\
		&\leq \frac{2\cdot\biggl(\sum_{h\in H}\biggl(\#\text{inputs of $h$ in $S$}-\#\text{outputs of $h$ in $S$}\biggr)^2\biggr)}{\vol S}\\
		&=2\cdot \tilde{\nu}(S).
	\end{align*}Since this is true for all such $S$, and since $\vol S\leq \vol \bar{S}$ if and only if $\vol S\leq \vol V/2$,
	\begin{equation*}
	    \lambda_{\min}\leq 2\cdot \min_{\substack{\emptyset\neq S\subseteq V\,:\\\vol S\leq \vol \bar{S}}}\tilde{\nu}(S)=2\tilde{h}.
	\end{equation*}
    \end{proof}

		\begin{lemma}\label{lemma:upper_m=0}If $m_V=0$,
		\begin{equation*}
		    \lambda_{\min}\leq \tilde{h}.
		\end{equation*}
		\end{lemma}
		\begin{proof}Fix $S$ that minimizes $\nu(S)$ and let $f_S:V\rightarrow\mathbb{R}$ be $1$ on $S$ and $0$ otherwise. Then,
		\begin{align*}
		    \lambda_{\min}=\lambda_1\leq \RQ(f_S)=\tilde{\nu}(S)=\tilde{h}.
		\end{align*}	\end{proof}
	\begin{remark}If we compare Lemma \ref{lemma:upper} and Lemma \ref{lemma:upper_m=0} we can observe the following. The upper bound $2\tilde{h}$ in Lemma \ref{lemma:upper} has a multiplication by $2$ coming from the fact that $\lambda_1=0$ and, in particular, coming from the fact that we must impose orthogonality to the constants. When we don't need to impose any orthogonality, i.e. in the case $m_V=0$, the upper bound can be simply $\tilde{h}$. We can therefore expect that, if $\lambda_{\min}=\lambda_k$, the orthogonality to $k-1$ different harmonic functions brings to $k-1$ constrains on $f:V\rightarrow\mathbb{R}$ and therefore we may have something like $\lambda_{\min}=\lambda_k\leq F(k)\cdot \tilde{h}.$
	
		\end{remark}
		\subsection{Cheeger lower bounds}
		\begin{remark}\label{remark:h-weakdel}Let $S\subset V$ be a minimizer for $\tilde{\nu}(S)$ and consider a weak vertex addition $\Gamma\cup\{\hat{v}\}$. Then,
		\begin{equation*}
		    \vol S(\Gamma\cup\{\hat{v}\})=\vol S(\Gamma)\leq \frac{1}{2}\vol V\leq \frac{1}{2}\vol (V\cup \hat{v})
		\end{equation*}and
		\begin{equation}\label{eq:hweakvertex}
		   \tilde{h}(\Gamma)=\tilde{\nu}(S)(\Gamma)= \tilde{\nu}(S)(\Gamma\cup\{\hat{v}\})\geq \tilde{h}(\Gamma\cup\{\hat{v}\}).
		\end{equation}Therefore, a weak vertex addition brings to a non-increasing $\tilde{h}$.\end{remark}
		\begin{lemma}If $\lambda_{\min}=\lambda_k$ and there exists a graph $G$ that can be obtained from $\Gamma$ by a weak deletion of $r$ vertices, where $r\leq k-2$, then
		\begin{equation*}
		    \frac{1}{2}\tilde{h}^2\leq \lambda_{\min}.
		\end{equation*}
		\end{lemma}
		\begin{proof}By \eqref{eq:hweakvertex}, \eqref{eq:Cheegergraphs} and Lemma \ref{lemma:Cauchy},
		\begin{equation*}
		    \frac{1}{2}\tilde{h}^2\leq\frac{1}{2}\tilde{h}(G)^2\leq \lambda_2(G)\leq \lambda_{k-r}(G)\leq \lambda_{k}(\Gamma)=\lambda_{\min}.
		\end{equation*}
		\end{proof}

	We now prove a generalization of the Cheeger lower bound for a particular class of oriented hypergraphs. Namely, we fix a hypergraph $\Gamma=(V,H)$ such that, for each $h\in H$,
		\begin{equation*}
        \# h_{in}=\# h_{out}=:c
    \end{equation*}is constant and does not depend on $h$. Since the number of inputs equals the number of outputs, in each hyperedge we can couple each input with exactly one output. In this way we get a graph $G$ that we call \emph{a underlying graph of $\Gamma$.}
    \begin{lemma}\label{lemma:underlying}If $\lambda_{\min}=\lambda_k$ and\begin{equation*}
        \# h_{in}=\# h_{out}=:c
    \end{equation*}is constant for each $h\in H$, then $\lambda_{k-1}(G)=0$ for each underlying graph $G$ of $\Gamma$. Furthermore, if there exists an underlying graph $G$ with $\lambda_k(G)>0$, then
    \begin{equation}\label{eq:lower}
        \lambda_{\min}\geq \frac{1}{2c}\tilde{h}^2.
    \end{equation}
    \end{lemma}
    
    \begin{proof}
    Fix any underlying graph $G=(V,E)$ of $\Gamma$. Observe that
    \begin{equation*}
        \deg_{\Gamma} (i)=\deg_G (i) \text{ for each }i\in V
    \end{equation*}and, for each hyperedge $h$ of $\Gamma$, $G$ has $c$ edges. Therefore $\#E=c\cdot \# H$. Also, we can see a function $\gamma:H\rightarrow\mathbb{R}$ as a function $\gamma:E\rightarrow\mathbb{R}$ that is equal to $\gamma(h)$ on each edge $e$ coming from $h$. If $\tau:H\rightarrow\mathbb{R}$ is an harmonic function for $\Gamma$,
    \begin{equation*}
       \sum_{h'\,:\, i\text{ input}}\tau(h')-\sum_{h''\,:\, i\text{ output}}\tau(h'')=0
       \end{equation*}for each $i\in V$. Therefore, $\tau$ is an harmonic function for $G$ as well. Also, if $\gamma:H\rightarrow\mathbb{R}$ is orthogonal to such an harmonic function in $\Gamma$, then $\sum_{h\in H}\gamma(h) \tau(h)=0$. Therefore, also
       \begin{equation*}
           \sum_{e\in E} \gamma(e) \tau(e)=\sum_{h\in H}c\cdot \gamma(h) \tau(h)=0
       \end{equation*}and $\gamma$ is also orthogonal to $\tau$ in $G$. This implies that $\lambda_{k-1}(G)=0$.

    Now, let $\gamma:H\rightarrow\mathbb{R}$ be an eigenfunction for $\lambda_k=\lambda_{\min}(\Gamma)$. By the remarks above, $\gamma$ is orthogonal to the constants also in $G$. Also, assume that $\lambda_k(G)>0$, so that $\lambda_{\min}(G)=\lambda_k(G)$. Then,
    \begin{align*}
        \lambda_k&=\lambda_{\min}(\Gamma)=\frac{\sum_{i\in V}\frac{1}{\deg (i)}\cdot \biggl(\sum_{h_{\text{in}}: i\text{ input}}\gamma(h_{\text{in}})-\sum_{h_{\text{out}}: v\text{ output}}\gamma(h_{\text{out}})\biggr)^2}{\sum_{h\in H}\gamma(h)^2}\\
        &=c\cdot \frac{\sum_{i\in V}\frac{1}{\deg (i)}\cdot \biggl(\sum_{h_{\text{in}}: i\text{ input}}\gamma(h_{\text{in}})-\sum_{h_{\text{out}}: i\text{ output}}\gamma(h_{\text{out}})\biggr)^2}{\sum_{h\in H}c\cdot \gamma(h)^2}\\
        &\geq c \cdot \lambda_{k}(G)\\
        &\geq c\cdot h(G)^2.
    \end{align*}
 Also, given $S$ that minimizes $\nu(S)(G)$, let $\bar{E}\subset E$ be the set of edges $e$ such that
    \begin{equation*}
        \#\text{inputs of $e$ in $S$}-\#\text{outputs of $e$ in $S$}> 0
    \end{equation*}and let $\bar{H}\subset H$ be the set of hyperedges corresponding to the edges in $\bar{E}$. Then,
    \begin{align*}
        \tilde{h}(\Gamma)&\leq \tilde{\nu}(S)(\Gamma)=\frac{\sum_{h\in H}\biggl(\#\text{inputs of $h$ in $S$}-\#\text{outputs of $h$ in $S$}\biggr)^2}{\vol S}\\
        &\leq \frac{\#\bar{H}\cdot c^2}{\vol S}=c\cdot\frac{ \#\bar{E}}{\vol S}= c\cdot \nu(S)(G)=\frac{c}{2} h(G).
    \end{align*}Therefore,
    \begin{equation*}
        \lambda_{\min}(\Gamma)\geq \frac{c}{2} h(G)^2\geq \frac{1}{2c} \tilde{h}(\Gamma)^2.
    \end{equation*}	
		\end{proof}
		\begin{remark}Lemma \ref{lemma:underlying} can be applied to the case where $\Gamma$ is a graph, by taking as underlying graph $\Gamma$ itself. In this case, $c=1$, therefore \eqref{eq:lower} coincides with the usual Cheeger lower bound. 
    \end{remark}
\section{General bounds}\label{section:generalbounds}
We prove some general characterizations and bounds for $\lambda_n$ and $\lambda_{\min}$ that do not involve $\tilde{h}$.

\begin{lemma}\label{lemma:main-small-large}
Given $i\in V$ and $h\in H$, let $\mathcal{I}_i:H\to\R$ and $\mathcal{I}^h:V\to\R$ be defined by $\mathcal{I}_i(h):=\mathcal{I}^h(i):=\mathcal{I}_{ih}$. Then,
\begin{equation}\label{eq:lambda-min}
    \lambda_{\min}=\min\limits_{\gamma\in \mathrm{span}\{ \mathcal{I}_i\,:\,i\in V\}}\frac{\sum_{i\in V}\frac{1}{\deg(i)}( \mathcal{I}_i,\gamma )_H^2}{ (\gamma,\gamma )_H}=\min\limits_{f\in\mathrm{span}\{D^{-\frac12}\mathcal{I}^h\,:\,h\in H\}}\frac{\sum_{h\in H}\langle D^{-\frac12}\mathcal{I}^h,f \rangle^2}{\langle f,f\rangle}
\end{equation}and
\begin{equation}\label{eq:lambda-max}
    \lambda_{n}=\max\limits_{\gamma\in \mathrm{span}\{ \mathcal{I}_i\,:\,i\in V\}}\frac{\sum_{i\in V}\frac{1}{\deg(i)}( \mathcal{I}_i,\gamma )_H^2}{ (\gamma,\gamma )_H}=\max\limits_{f\in\mathrm{span}\{D^{-\frac12}\mathcal{I}^h\,:\,h\in H\}}\frac{\sum_{h\in H}\langle D^{-\frac12}\mathcal{I}^h,f \rangle^2}{\langle f,f\rangle},
\end{equation}where, for functions on the vertex set, $(\cdot,\cdot )$ is the scalar product in Definition \ref{def:scalarfg} and $\langle\cdot,\cdot \rangle$ is the scalar product without the weight $(\deg(1),\ldots,\deg(n))$. For functions on the hyperedge set, $(\cdot,\cdot )_H$ is the scalar product in Definition \ref{def:scalarproductgamma}.
\end{lemma}
\begin{proof}
We prove \eqref{eq:lambda-min}, the proof of \eqref{eq:lambda-max} being similar. Observe that
$$
\frac{\sum_{h\in H}\left(\sum_{i\in h_{in}}f(i)-\sum_{i\in h_{out}}f(i)\right)^2}{\sum_{i\in V} \deg(i) f(i)^2}=\frac{\sum_{h\in H}\langle\mathcal{I}^h,f \rangle^2}{(f,f)}.
$$
Therefore, $f$ is an eigenfunction corresponding to the eigenvalue $0$ if and only if $\langle\mathcal{I}^h,f \rangle=0$ for all $h\in H$. Hence, the linear space of all harmonic functions is the orthogonal complement of the set of functions $\mathcal{I}^h$: $(\mathrm{span}\{\mathcal{I}^h:h\in H\})^\bot$. By the min-max principle,
\begin{align*}
\lambda_{\min}&=\min\limits_{\substack{(f,g)=0\\\forall g\in (\mathrm{span}\{\mathcal{I}^h:h\in H\})^\bot}} \frac{\sum_{h\in H}\langle\mathcal{I}^h,f \rangle^2}{(f,f)}
\\&= 
\min\limits_{f\in \mathrm{span}\{D^{-1}\mathcal{I}^h:h\in H\}}\frac{\sum_{h\in H}\langle\mathcal{I}^h,f \rangle^2}{(f,f)}
\\&=
\min\limits_{f\in\mathrm{span}\{D^{-\frac12}\mathcal{I}^h:h\in H\}}\frac{\sum_{h\in H}\langle D^{-\frac12}\mathcal{I}^h,f \rangle^2}{\langle f,f\rangle}.
\end{align*}
\end{proof}

\begin{corollary}
\label{thm:min-quantity-max}
The following quantities are all no less than $\lambda_{\min}$ and no larger than $\lambda_n$:
\begin{enumerate}[({C}1)]
    
    \item $$1+\sum_{j\in V}\frac{1}{\deg(i)\deg(j)}A_{ij}^2$$ for any $i\in V$;
    \item  $$1+\frac{1}{n}\sum_{j\in V}\sum_{i\in V}\frac{1}{\deg(j)\deg(i)}A_{ij}^2$$
    
    \item $$1+\frac{\sum_{j\in V}\frac{1}{\deg(j)}\sum_{i\in V}A_{ij}^2}{\sum_{j\in V}\deg(j)}.$$
\end{enumerate}Moreover, $\lambda_{\min}$ equals (C2) or (C3) if and only if 
\begin{equation*}
    1+\sum_{j\in V}\frac{1}{\deg(i)\deg(j)}A_{ij}^2=\lambda_{\min}
\end{equation*}
 for all $i\in V$. The same holds for $\lambda_n$. 
\end{corollary}
\begin{proof}
In order to prove (C1), for each $i\in V$ we consider the Rayleigh Quotient $\gamma=\mathcal{I}_i$ and we apply Lemma \ref{lemma:main-small-large}.
Furthermore, (C2) is the arithmetic mean of the constants in (C1) over $i\in V$, while (C3) is the weighted arithmetic mean of the constants in (C1) with the weights $\deg(1),\ldots,\deg(n)$. This proves (C2) and (C3).
\end{proof}

\begin{definition}Given $n,d\in\mathbb{N}$ with $n\ge d\ge 1$, let 
$$C_{n,d}:=\min\limits_{\vec v_1,\ldots,\vec v_n\in \mathbb{S}^{d-1}}\max\limits_{\varepsilon_1,\ldots,\varepsilon_n\in\{-1,1\}}\|\varepsilon_1\vec v_1+\ldots+\varepsilon_n\vec v_n\|_2,$$
where $\mathbb{S}^{d-1}$ is the unit sphere of dimension $(d-1)$ in $\R^d$ and  $\|\cdot\|_2$ is the standard Euclidean norm in $\R^d$. 
\end{definition}
The quantity $C_{n,d}$ is a known geometric constant \cite{Cnd1,Cnd2,Cnd3,Cnd4} that characterizes the best lower bound of the diameter of the $d$-dimensional zonotope $[-\frac12 \vec v_1,\frac12 \vec v_1]+\ldots+[-\frac12 \vec v_n,\frac12 \vec v_n]$ generated by $n$ unit vectors, where the summation here is the Minkowski sum of a finite number of segments. Since zonotopes have many interesting geometric properties in the theory of polyhedron  \cite{Polytope-Ziegler}, 
the constant $C_{n,d}$ is studied in discrete geometry. Interestingly, the unconstrained quadratic maximization in zero-one variables has been equivalently transformed into
\begin{equation*}
    \max\limits_{\varepsilon_1,\ldots,\varepsilon_n\in\{-1,1\}}\|\varepsilon_1\vec v_1+\ldots+\varepsilon_n\vec v_n\|_2^2
\end{equation*}
 for certain $\vec v_1,\ldots,\vec v_n$, bridging mathematical  optimizations with zonotopes \cite{Cnd4}. Also, it is known that $$C_{n,n}=\sqrt{n}<C_{n,d}$$ for all $d<n$.

The following proposition shows a relation between this constant and the largest eigenvalue of the normalized Laplacian. To the best of our knowledge, this result is new also for graphs.

\begin{proposition}
$$\lambda_n\ge \frac1n \cdot C_{n,n-m_V}^2.$$
\end{proposition}

\begin{proof}
The proof is based on Lemma \ref{lemma:main-small-large}. Given $i\in V$, let $\vec v_i:=\frac{1}{\sqrt{\deg(i)}}\mathcal{I}_i$ and let $\varepsilon_1,\ldots,\varepsilon_n\in\{-1,1\}$ such that $$ \|\varepsilon_1\vec v_1+\ldots+\varepsilon_n\vec v_n\|_2^2=\max\limits_{\varepsilon_1',\ldots,\varepsilon_n'\in\{-1,1\}}\|\varepsilon_1'\vec v_1+\ldots+\varepsilon_n'\vec v_n\|_2^2.$$  Let also $$\gamma:=\frac{\varepsilon_1\vec v_1+\ldots+\varepsilon_n\vec v_n}{\|\varepsilon_1\vec v_1+\ldots+\varepsilon_n\vec v_n\|_2}.$$ Then, by Lemma \ref{lemma:main-small-large} and by the fact that $\dim(\mathrm{span}(\vec v_1,\ldots,\vec v_n))=n-m_V$,  \begin{align*}
    \lambda_n&\ge \sum_{i=1}^n \langle \vec v_i,\gamma\rangle^2 =\sum_{i=1}^n |\varepsilon_i|^2\langle \vec v_i,\gamma\rangle^2
    \ge \frac1n\left(\sum_{i=1}^n \varepsilon_i\langle \vec v_i,\gamma\rangle\right)^2 \\&=\frac1n\left(\langle \sum_{i=1}^n \varepsilon_i\vec v_i,\gamma\rangle\right)^2  = \frac1n \|\varepsilon_1\vec v_1+\ldots+\varepsilon_n\vec v_n\|_2^2\ge \frac1n \cdot C_{n,n-m_V}^2.
\end{align*}
\end{proof}

We conclude this section by proving a general upper bound for $\lambda_{\min}$ and lower bound for $\lambda_n$ that only involves the multiplicity of $0$.
\begin{theorem}\label{thm:min-bound-max}
\begin{equation*}
\lambda_{\min}\le \frac{n}{n-m_V}  \le \lambda_n
\end{equation*}and one of them is an equality if and only if $\lambda_{\min}=\lambda_n$.
\end{theorem}
\begin{proof}
By Remark \ref{remark:trace}, $\sum_{i=m_V+1}^n\lambda_i=n$. Hence, since there are exactly $n-m_V$ non-zero eigenvalues $\lambda_i$ with $\lambda_{\min}\leq \lambda_i$, $$(n-m_V)\lambda_{\min}\leq \sum_{i=m_V+1}^n\lambda_i=n,$$ that is, $\lambda_{\min}\leq n/(n-m_V)$, with equality if and only if $\lambda_{\min}=\lambda_n$. Similarly, one can see that $\lambda_n\geq n/(n-m_V)$, with equality if and only if $\lambda_{\min}=\lambda_n$.
\end{proof}

\begin{ex}A hypergraph $\Gamma$ is such that the inequalities in Theorem \ref{thm:min-bound-max} are equalities with $m_V=0$ if and only if\begin{equation*}
    \lambda_1=\ldots=\lambda_n=1,
\end{equation*}therefore if and only if $L=\id$, which happens if and only if $A_{ij}=0$ for all $i\neq j\in V$, i.e. if and only if, for all $i\neq j$,
\begin{align*}
&\# \{\text{hyperedges in which }i \text{ and }j\text{ are anti-oriented}\}\\
&=\# \{\text{hyperedges in which }i \text{ and }j\text{ are co-oriented}\}.
\end{align*}This is the case, for instance, if $\Gamma$ is a hypergraph on $n$ vertices and $n$ hyperedges such that each hyperedge contains exactly one vertex.
\end{ex}

\begin{ex}
Let $\Gamma$ be a hypergraph for which the inequalities in Theorem \ref{thm:min-bound-max} are equalities, with $m_V=n-1$. Then,
\begin{equation}\label{eq:exk=n-1}
    \lambda_{1}=\ldots=\lambda_{n-1}=0\quad\text{and}\quad \lambda_n=n.
\end{equation}
Since we always have $\sum_{i=1}^n\lambda_i=n$, $\lambda_n=n$ implies \eqref{eq:exk=n-1}. By \cite[Corollary 2]{Sharp}, this happens if and only if each hyperedge contains all vertices.
\end{ex}

\begin{ex}
If $\Gamma$ is given by the union of $r$ copies of the complete graph $K_{n/r}$, then $m_V=r$ and all non-zero eigenvalues are $$\lambda=\frac{n/r}{n/r-1}=\frac{n}{n-r},$$ with multiplicity $n-r$. \end{ex}
\begin{remark}
By \cite[Lemma 4.2]{Sharp}, $\lambda_n\le \max\limits_{h\in H}|h|$. Together with Theorem \ref{thm:min-bound-max}, this implies that $$\frac{n}{n-m_V}\leq \max\limits_{h\in H}|h|.$$ Therefore, $$m_V\le n\left(1-\frac{1}{\max\limits_{h\in H}|h|}\right).$$ In the case of graphs, this says that $m_V\le n/2$ and it is immediate to check, since we are assuming that there are no isolated vertices.
\end{remark}

\section{Coloring number}\label{section:coloring}
We now generalize the notion of coloring number and we show that it is related to the eigenvalues of $L$.
\begin{definition}A \emph{proper $k$-coloring of the vertices} is a function $f:V\to \{1,\ldots,k\}$ such that $f(i)\ne f(j)$ for all $i\ne j\in h$ and for all $h\in H$. The \emph{vertex coloring number} of $\Gamma$, denoted $\chi(\Gamma)$, is the minimal $k$ such that there exists a proper $k$-coloring.
\end{definition}
\begin{remark}
Observe that the coloring number of an oriented hypergraph $\Gamma=(V,H)$ equals the coloring number of a graph $G=(V,E)$ that has the same vertices as $\Gamma$ and has, instead of each hyperedge $h$, a complete sub-graph $K_{|h|}$. Hence, the problem of computing the coloring number of a hypergraph reduces to the graph case.
\end{remark}
\begin{remark}
If $\hat{\Gamma}$ is obtained from $\Gamma$ by deleting vertices, then $\chi(\hat{\Gamma})\leq \chi(\Gamma)$.
\end{remark}

\begin{theorem}\label{thm:coloring-main}
For any oriented hypergraph $\Gamma$,
\begin{equation*}\lambda_n\ge 1+\left(1-\frac{\tilde{e}(S)}{\vol(S)}\right)\frac{1}{\chi(S)-1}\ge\lambda_1,\quad \forall S\subseteq V, S\neq \emptyset.\end{equation*}

 In particular, \begin{equation*}\lambda_n\ge\frac{\chi(\Gamma)}{\chi(\Gamma)-1}-\frac{\tilde{e}(V)}{\vol(V)}\cdot\frac{1}{\chi(\Gamma)-1}\ge \lambda_1\quad\text{ and }\quad\lambda_n\ge\frac{\chi-\tilde{h}'}{\chi-1},\end{equation*} where $\tilde{h}':=\min\limits_{\emptyset \neq S\subseteq V}\frac{\tilde{e}(S)}{\vol(S)}$.
\end{theorem}
\begin{proof}
Let $\chi:=\chi(\Gamma)$ and let $V_1,\ldots,V_\chi$ be the coloring classes of $V$. Given $k\in \{1,\ldots,\chi\}$, define a function $f:V\to\R$ by
\begin{equation*}f(i):=\begin{cases}t&\text{ if }i\in V_k,\\
1&\text{ if }i\not\in V_k.\end{cases}\end{equation*}
Since $V_k\cap h$ has at most one element for all $h\in H$,
\begin{equation*}\left(\sum_{j\in h_{in}} f(j)-\sum_{j'\in h_{out}} f(j')\right)^2=\begin{cases}(t+\#h_{in}-1-\#h_{out})^2&\text{ if } V_k\cap h_{in} \ne\emptyset,\\
(t+\#h_{out}-1-\#h_{in})^2&\text{ if } V_k\cap h_{out}\ne\emptyset ,\\
(\#h_{out}-\#h_{in})^2&\text{ otherwise}.\end{cases}\end{equation*}
Hence,
\begin{align*}
\sum_{h\in H}\left(\sum_{j\in h_{in}} f(j)-\sum_{j'\in h_{out}} f(j')\right)^2=\;& \sum_{h_{in}\cap V_k\ne\emptyset}(t+\#h_{in}-1-\#h_{out})^2
\\&+\sum_{h_{out}\cap V_k\ne\emptyset}(t+\#h_{out}-1-\#h_{in})^2
\\&+\sum_{h\cap V_k=\emptyset}(\#h_{out}-\#h_{in})^2\end{align*}
and
\begin{equation*}
 \sum_{i\in V}\deg(i)f(i)^2=t^2\sum_{i\in V_k}\deg(i)+\sum_{i\not\in V_k}\deg(i).
\end{equation*}
Since
\begin{equation*}\lambda_n\ge\frac{\sum_{h\in H}\left(\sum_{j\in h_{in}} f(j)-\sum_{j'\in h_{out}} f(j')\right)^2}{\sum_{i\in V}\deg(i)f(i)^2}\ge \lambda_1,\end{equation*}
by taking the summation of the above inequalities over all $k$'s we obtain
\begin{align*}
\lambda_n\sum_{k=1}^\chi\left(t^2\sum_{i\in V_k}\deg(i)+\sum_{i\not\in V_k}\deg(i)\right)\ge\;& \sum_{k=1}^\chi\sum_{h_{in}\cap V_k\ne\emptyset}(t+\#h_{in}-1-\#h_{out})^2
\\&+\sum_{k=1}^\chi\sum_{h_{out}\cap V_k\ne\emptyset}(t+\#h_{out}-1-\#h_{in})^2
\\&+\sum_{k=1}^\chi\sum_{h\cap V_k=\emptyset}(\#h_{out}-\#h_{in})^2.
\end{align*}
This can be simplified as
 \begin{align*}
&\lambda_n\left(\sum_{i\in V}\deg(i)(t^2-1)+\chi \sum_{i\in V}\deg(i)\right)\\
\ge\;& \sum_{h\in H} \#h_{in}(t+\#h_{in}-1-\#h_{out})^2
\\&+\sum_{h\in H}\#h_{out}(t+\#h_{out}-1-\#h_{in})^2
\\&+\sum_{h\in H} (\chi-\#h_{in}-\#h_{out}) (\#h_{out}-\#h_{in})^2
\\=\;&(t-1)^2\sum_{h\in H}(\#h_{in}+\#h_{out})
\\&+(2(t-1)+\chi)\sum_{h\in H}(\#h_{out}-\#h_{in})^2.
\end{align*}
Since $\vol(V)=\sum_{h\in H}(\#h_{in}+\#h_{out})$ and $\tilde{e}(V)=\sum_{h\in H}(\#h_{out}-\#h_{in})^2$,
\begin{equation*}
\lambda_n\vol(V)(t^2-1+\chi)\ge (t-1)^2\vol(V)+(2(t-1)+\chi)\tilde{e}(V),
\end{equation*}
hence
\begin{equation*}\lambda_n\ge \frac{(t-1)^2\vol(V)+(2(t-1)+\chi)\tilde{e}(V)}{\vol(V)(t^2-1+\chi)}=1+
\frac{\tilde{e}(V)-\vol(V)}{\vol(V)}\cdot\frac{2(t-1)+\chi}{t^2-1+\chi}.\end{equation*}
Now, using the fact that 
\begin{equation*}
    \max\limits_{t\in\R}\frac{2(t-1)+\chi}{t^2-1+\chi}=1 \text{ at }t=1\quad\text{and}\quad \min\limits_{t\in\R}\frac{2(t-1)+\chi}{t^2-1+\chi}=\frac{-1}{\chi-1}\text{ at }t=1-\chi,
\end{equation*}it follows that
\begin{equation}\label{eq:lambda-n-e(V)}
\lambda_n\ge \begin{cases}\frac{\tilde{e}(V)}{\vol(V)}&\text{ if } \tilde{e}(V)\ge \vol(V),\\
\frac{\chi}{\chi-1}-\frac{\tilde{e}(V)}{\vol(V)}\cdot\frac{1}{\chi-1}&\text{ if } \tilde{e}(V)< \vol(V).\end{cases}
\end{equation}
In summary, \begin{equation*}\left\{\frac{\chi}{\chi-1}-\frac{\tilde{e}(V)}{\vol(V)}\cdot\frac{1}{\chi-1},
\frac{\tilde{e}(V)}{\vol(V)}\right\}\subset [\lambda_1,\lambda_n].\end{equation*}

Now, given $V'\subset V$, let $H':=\{h\cap V':h\in H\}$. Then, $\Gamma':=(V',H')$ is the restricted sub-hypergraph of $H$ on $V'$.  By Lemma \ref{lemma:Cauchy}, \begin{equation*}\lambda_n(\Gamma)\ge \lambda_{\max}(\Gamma')\ge \lambda_1(\Gamma')\ge \lambda_1(\Gamma).\end{equation*} Consequently,
\begin{equation*}\left\{\left.\frac{\tilde{e}(S)}{\vol(S)},\;\;\frac{\chi(S)}{\chi(S)-1}-\frac{\tilde{e}(S)}{\vol(S)}\cdot\frac{1}{\chi(S)-1}
\;\right|\;S\subset V,S\ne\emptyset\right\}\subset [\lambda_1(\Gamma),\lambda_n(\Gamma)].\end{equation*}
This completes the proof.
\end{proof}

\begin{corollary}\label{cor:coloring-c-difference}
Let $\Gamma$ be an oriented hypergraph such that $|\#h_{in}-\#h_{out}|=c$ for all $h\in H$, for some $c\geq 0$. Then,
\begin{equation}\label{eq:cor:coloring-c-difference1}
\lambda_n\ge 1+\frac{\vol(V)-c^2\cdot m}{\vol(V)}\cdot\frac{1}{\chi(V)-1}\ge \lambda_1,
\end{equation}where $m:=\# H$. If we further assume that $\Gamma$ is $r$--uniform (where $r=c+2l$ for some $l\in \mathbb{Z}_{\ge0}$), then 
\begin{equation}\label{eq:cor:coloring-c-difference2}
\lambda_n\ge \frac{\chi}{\chi-1}-\frac{c^2}{r}\cdot\frac{1}{\chi-1}\ge \lambda_1.
\end{equation}
\end{corollary}

\begin{proof}
Since $|\#h_{in}-\#h_{out}|=c$ for all $h\in H$, $\tilde{e}(V)=c^2\cdot m$. By Theorem \ref{thm:coloring-main}, we get \eqref{eq:cor:coloring-c-difference1}. The further assumption that $\Gamma$ is $r$--uniform implies $r\cdot m=\vol(V)$. Thus, $$\frac{\tilde{e}(V)}{\vol(V)}=\frac{c^2}{r}$$ and we immediately obtain \eqref{eq:cor:coloring-c-difference2}.
\end{proof}

\begin{corollary}\label{cor:coloring-c-signless}
Let $\Gamma=(V,H)$ be an oriented hypergraph such that $\#h_{in}=c$ and $\#h_{out}=0$ for all $h\in H$, for some $c\in \mathbb{N}$. Then, 
\begin{equation}\label{eq:cor:coloring-c-signless}
 \lambda_1\leq \frac{\chi(\Gamma)-c}{\chi(\Gamma)-1}\quad\text{and} \quad \lambda_n=c.
\end{equation}If, in addition, $\Gamma$ is $c$--complete, then $\chi(\Gamma)=n$ and
\begin{equation*}
    \lambda_1=\ldots=\lambda_{n-1}=\frac{n-c}{n-1}.
\end{equation*}
\end{corollary}

\begin{proof}
By construction, $\Gamma$ is a $c$--uniform, bipartite hypergraph. Therefore, by \cite[Lemma 4.2]{Sharp}, $\lambda_n=c$. Also, by Corollary \ref{cor:coloring-c-difference} with $r=c$, we get that $$\lambda_1\le\frac{\chi-c}{\chi-1}.$$
If, in addition, $\Gamma$ is $c$--complete, clearly $\chi(\Gamma)=n$ and each vertex has degree ${n-1\choose c-1}$. Now, for each vertex $k$, let $f_k:V\to\R$ be defined by $f_k(k):=n-1$ and $f_k(i):=-1$ for $i\neq k$. Then,
$$L f_k(k) = \frac{1}{\deg(k)}\sum_{h\ni k}(n-1+(c-1)(-1))=n-c = \frac{n-c}{n-1}\cdot f_k(k)$$ and for $i\neq k$,
\begin{align*} 
L f_k(i) &=\frac{1}{\deg(i)}\left(\sum_{h\ni i,h\not\ni k}(-c)+\sum_{h\supset \{i,k\}}(n-c)\right)\\
&= -c +\frac{1}{{n-1\choose c-1}}\cdot {n-2\choose c-2}\cdot n\\
&=\frac{n-c}{n-1}\cdot f_k(i).\end{align*}
Therefore $Lf_k=f_k\cdot ((n-c)/(n-1))$, which proves that $(n-c)/(n-1)$ is an eigenvalue and the functions $f_k$ are corresponding eigenfunctions. Now, since $$\dim(\mathrm{span}(f_1,\ldots,f_n))=n-1$$ and since $\lambda_n=c$, the multiplicity of $(n-c)/(n-1)$ is $n-1$. Hence,
$$\lambda_1=\ldots=\lambda_{n-1}=\frac{n-c}{n-1}.$$
\end{proof}

\begin{corollary}\label{cor:boundcoloring}
If $\#h_{in}=\#h_{out}$ for all $h\in H$,
\begin{equation*}\lambda_n\ge \frac{\chi(\Gamma)}{\chi(\Gamma)-1}=\lambda_n(K_\chi),\end{equation*}
where $K_\chi$ is the complete graph on $\chi$ vertices.
\end{corollary}

\begin{proof}
It follows directly by taking $c=0$ in Corollary \ref{cor:coloring-c-difference}. It is known that a complete graph $K_N$ on $N$ vertices has the maximal eigenvalue $N/(N-1)$.
\end{proof}

\begin{corollary}\label{cor:coloring1}
If there exists a hypergraph $\hat{\Gamma}$ obtained from $\Gamma$ by weak-vertex deletion of some vertices, such that 
\begin{equation*}
    \# h_{in}(\hat{\Gamma})=\# h_{out}(\hat{\Gamma})\quad\text{for each }h\in H,
\end{equation*}then
\begin{equation*}
    \lambda_n\ge \frac{\chi(\Gamma)}{\chi(\Gamma)-1}.
\end{equation*}
\end{corollary}
\begin{proof}
Assume that $\hat{\Gamma}$ is obtained from $\Gamma$ by weak-vertex deletion of $r$ vertices. By Lemma \ref{lemma:Cauchy} and Corollary \ref{cor:boundcoloring},
\begin{equation*}
\lambda_n(\Gamma)\geq\lambda_{n-r}(\hat{\Gamma})=\lambda_n(\hat{\Gamma})\geq \frac{\chi(\hat{\Gamma})}{\chi(\hat{\Gamma})-1}.
\end{equation*}Now, since $\chi(\hat{\Gamma})\leq \chi(\Gamma)$,
\begin{equation*}
    \frac{\chi(\hat{\Gamma})}{\chi(\hat{\Gamma})-1}\geq \frac{\chi(\Gamma)}{\chi(\Gamma)-1}.
\end{equation*}Hence,
\begin{equation*}
    \lambda_n\ge \frac{\chi(\Gamma)}{\chi(\Gamma)-1}.
\end{equation*}
\end{proof}
\begin{remark}
The results on the chromatic number above closely relate to the Hoffman's bound for graph, which states that, for a graph $G$,
\begin{equation*}
    \chi(G)\geq 1-\frac{\lambda_{\max}(A)}{\lambda_{\min}(A)},
\end{equation*}where $\lambda_{\max}(A)$ and $\lambda_{\min}(A)$ are the largest and the smallest eigenvalues of the adjacency matrix \cite{Hoffman1,Hoffman2}.
\end{remark}
In the setting of Corollary \ref{cor:boundcoloring}, when does $\lambda_n = \chi(\Gamma)/(\chi(\Gamma)-1)$ hold? The next proposition answers this question.

  \begin{proposition}\label{pro:sharp-chromatic}
  For a hypergraph $\Gamma$ with  $\#h_{in}=\#h_{out}$ for each $h$ and chromatic number $\chi$, the following are equivalent: 
  \begin{enumerate}
      \item $\lambda_n=\chi/(\chi-1)$.
      \item  The vertex set can be partitioned as $V= V_1\sqcup\ldots\sqcup V_\chi$ such that:
      
      \begin{itemize}
          \item $\#(h\cap V_k)\in\{0,1\}$, for all $h$ and for all $k$,
          \item $(\chi-1)|\deg(i)$ for all $i\in V$, and
          \item For all $k$ and for all $i\not\in V_k$,
          \begin{align*}
             &\#\{h\in H: i\text{ and }h\cap V_k\text{ anti-or.}\}-\#\{h\in H: i\text{ and }h\cap V_k\text{ co-or.}\}\\
             =&\frac{\deg(i)}{\chi-1}. 
          \end{align*}
        
      \end{itemize}   
  
  \end{enumerate}  
  \end{proposition}
  \begin{proof} We first show a direct proof of Corollary \ref{cor:boundcoloring}, which states that $\lambda_n\leq \chi/(\chi-1)$. Let $V_1,\ldots,V_\chi$ be the coloring classes of $V$. Given $k\in \{1,\ldots,\chi\}$, define a function $f_k:V\to\R$ by
\begin{equation*}f_k(i)=\begin{cases}\chi-1,&\text{ if }i\in V_k\\
-1&\text{ if }i\not\in V_k.\end{cases}\end{equation*}
Since $\#h_{in}=\#h_{out}$ and $V_k\cap h$ has at most one element for all $h\in H$,
\begin{equation*}\left(\sum_{j\in h_{in}} f_k(j)-\sum_{j'\in h_{out}} f_k(j')\right)^2=\begin{cases}\chi^2&\text{ if } V_k\cap h\ne\emptyset,\\
0&\text{ otherwise}.\end{cases}\end{equation*}
Hence,
\begin{align}\nonumber
\lambda_n&\ge\frac{\sum_{h\in H}\left(\sum_{j\in h_{in}} f_k(j)-\sum_{j'\in h_{out}} f_k(j')\right)^2}{\sum_{i\in V}\deg(i)f_k(i)^2}\\
&=\frac{\chi^2\sum_{\substack{h\in H:\\ h\cap V_k\ne\emptyset}}1}{\sum_{i\in V_k}\deg(i)(\chi-1)^2+\sum_{i\not\in V_k}\deg(i)}.\label{eq:f_k-1}\end{align}
That is, for all $k$,
\begin{equation}\label{eq:f_k-2}
\lambda_n\cdot \left(\sum_{i\in V_k}\deg(i)((\chi-1)^2-1)+\sum_{i\in V}\deg(i)\right)\ge\chi^2\sum_{i\in V_k}\deg(i).\end{equation}
Summing up the above inequalities for all $k$,
\begin{equation*}
    \lambda_n\cdot \left(\sum_{i\in V}\deg(i)((\chi-1)^2-1)+\chi \sum_{i\in V}\deg(i)\right)\ge \chi^2\sum_{i\in V}\deg(i),
\end{equation*} and thus $\lambda_n(\chi-1)\ge \chi$. 

Now, it is clear that $\lambda_n=\chi/(\chi-1)$ if and only if \eqref{eq:f_k-2} (equivalently \eqref{eq:f_k-1}) is an equality for all $k$. Therefore, $\lambda_n=\chi/(\chi-1)$ if and only if $f_k$ is an eigenfunction for $\lambda_n$ for each $k=1,\ldots,\chi$. This is the case if and only if $$L f_k = \frac{\chi}{\chi-1}  f_k,\quad k=1,\ldots,\chi$$ that is,

$$
L f_k(i) =\frac{1}{\deg(i)}\left(\sum\limits_{h_{in}\ni i}\chi-
\sum\limits_{h_{out}\ni i}(-\chi)\right)=\chi= \frac{\chi}{\chi-1}  f_k(i)$$
for all $k$ and for all $i\in V_k$. Thus, $\lambda_n=\chi/(\chi-1)$ is equivalent to
$$
L f_k(i) = \frac{\chi}{\chi-1}  f_k(i),\quad \text{for all $k$ and for all } i\not\in V_k.
$$
That is,
$$
\frac{1}{\deg(i)}\sum\limits_{h\ni i,\,h\cap V_k\ne\emptyset}\left(\sum\limits_{i\text{ and }h\cap V_k\text{ co-or.}}\chi+\sum\limits_{i\text{ and }h\cap V_k\text{ anti-or.}}(-\chi)\right)=-\frac{\chi}{\chi-1},
$$
which is equivalent to
 $$\#\{h\in H: i\text{ and }h\cap V_k\text{ anti-or.}\}-\#\{h\in H: i\text{ and }h\cap V_k\text{ co-or.}\}=\frac{\deg(i)}{\chi-1}.$$
 In particular, $(\chi-1)|\deg(i)$ for all $i$. 
  \end{proof}

  \begin{corollary}\label{cor:c-complete}
  Given two natural numbers $n\ge 2$ and $c\ge 1$, let $\Gamma:=(V,H)$ be the $2c$--complete hypergraph defined by $V:=\{1,\ldots,n\}$ and $$H:=\{(h_{in},h_{out}): \#h_{in}=\#h_{out}=c,\, h_{in}\cap h_{out}=\emptyset\}.$$ Then, $\lambda_1=0$ and $\lambda_2=\ldots=\lambda_n=n/(n-1)$.
  \end{corollary}
  
  \begin{proof}
Clearly, $\chi(\Gamma)=n$. Now, for each $k=1,\ldots,n$, let $V_k:=\{k\}$. Then, for all $i$ and $k$:
$$\deg(i)={n-1\choose 2c-1}\frac{{2c\choose c}}{2},$$
$$\#\{h\in H: i\text{ and }h\cap V_k\text{ anti-oriented}\}={n-2\choose 2c-2}\frac{{2c\choose c}}{2}\frac{c}{2c-1},$$
$$\#\{h\in H: i\text{ and }h\cap V_k\text{ co-oriented}\}={n-2\choose 2c-2}\frac{{2c\choose c}}{2}\frac{c-1}{2c-1}.$$
Thus, $(n-1)|\deg(i)$ for all $i$, and 
\begin{align*}
   & \#\{h\in H: i\text{ and }h\cap V_k\text{ anti-or.}\}-\#\{h\in H: i\text{ and }h\cap V_k\text{ co-or.}\}
   \\=~& {n-2\choose 2c-2}\frac{{2c\choose c}}{2}\frac{1}{2c-1}= \frac{1}{n-1} {n-1\choose 2c-1}\frac{{2c\choose c}}{2}
    =\frac{\deg(i)}{n-1}.
\end{align*}
By Proposition \ref{pro:sharp-chromatic}, $\lambda_n=n/(n-1)$. 

Now, observe that $f=1$ is such that $Lf=0$, which means that $\lambda_1=0$ and $f$ is a corresponding eigenfunction. Therefore,
$$n=\sum_{i=1}^n\lambda_i=\sum_{i=2}^n\lambda_i\le (n-1)\lambda_n=n,$$
which implies that $\lambda_2=\ldots=\lambda_n=n/(n-1)$.
  \end{proof}
    \begin{remark}
      For $c=1$, Corollary \ref{cor:c-complete} gives a complete graph of order $n$. For $c\ge 2$, Corollary \ref{cor:c-complete} gives a $2c$--uniform and $\frac12{n-1\choose 2c-1}{2c\choose c}$--regular oriented hypergraph.
    \end{remark}

    \section{Cartesian product of hypergraphs}\label{section:unnormalized}
    In this last section, we define the Cartesian product of hypergraphs and we study the spectrum of the unnormalized Laplacian in this case.
    \begin{definition}
Given two oriented hypergraphs $\Gamma_1=(V_1,H_1)$ and $\Gamma_2=(V_2,H_2)$, their \emph{Cartesian product} $\Gamma:=(V,H)$, denoted $\Gamma_1\,\square\, \Gamma_2$, is defined by letting $V:=V_1\times V_2$ and 
\begin{align*}
    H:=\{&h: h_{in}=\{v\}\times h^2_{in}\, ,\, h_{out}=\{v\}\times h^2_{out}\,\text{ or }\, h_{in}= h^1_{in}\times \{u\},\, \\ &h_{out}=h^1_{out}\times \{u\}, \,
    \text{for some }v\in V_1,u\in V_2, h^i\in H_i\}.
\end{align*}
\end{definition}

\begin{proposition}Let 
\begin{equation*}
    \lambda_1\leq \ldots\leq \lambda_{n_1}\quad\text{and}\quad \mu_1\leq \ldots\leq \mu_{n_2}
\end{equation*}be the spectra of the unnormalized Laplacians for $\Gamma_1=(V_1,H_1)$ and $\Gamma_2=(V_2,H_2)$, respectively, where $n_1=\#V_1$ and $n_2=\#V_2$. Then, the spectrum of the unnormalized Laplacian of the Cartesian product $\Gamma_1\,\square\, \Gamma_2$ is given by
\begin{equation*}
    \lambda_i+\mu_j, \,\text{ for } \,i=1,\ldots,n_1\,\text{ and }\, j=1,\ldots,n_2.
\end{equation*}
\end{proposition}

\begin{proof}
Denote by $\Delta(\Gamma_i)$ the unnormalized Laplacian of $\Gamma_i$, for $i=1,2$. Let $f$ be an eigenfunction of $\Delta(\Gamma_1)$ with eigenvalue $\lambda$, that is, $\Delta(\Gamma_1) f=\lambda f$. Similarly, fix $g$ and $\mu$ such that $\Delta(\Gamma_2) g=\mu g$. According to Definition \ref{def:unnormalized-Laplacian},
\begin{equation*}\sum_{h_{in}^1\ni i}\left(\sum_{i'\in h_{in}^1}f(i')-\sum_{i''\in h_{out}^1}f(i'')\right)-\sum_{h_{out}^1\ni i}\left(\sum_{i'\in h_{in}^1}f(i')-\sum_{i''\in h_{out}^1}f(i'')\right)=\lambda f(i)\end{equation*} 
and 
\begin{equation*}\sum_{h_{in}^2\ni j}\left(\sum_{j'\in h_{in}^2}g(j')-\sum_{j''\in h_{out}^2}g(j'')\right)-\sum_{h_{out}^2\ni j}\left(\sum_{j'\in h_{in}^2}g(j')-\sum_{j''\in h_{out}^2}g(j'')\right)=\mu g(j)\end{equation*} 
for all $i=1,\ldots,n_1$ and $j=1,\ldots,n_2$. Now, let $f\otimes g: V_1\times V_2\rightarrow \R$ be defined by
\begin{equation*}
    f\otimes g\, (i,j):=f(i)\cdot g(j).
\end{equation*}Then, for all $(i,j)\in V_1\times V_2$, 
\begin{small}
\begin{align*}
&\Delta (f\otimes g)(i,j)\\
=&
\sum_{h_{in}\ni (i,j)}\left(\sum_{(i',j')\in h_{in}}f\otimes g(i',j')
-\sum_{(i',j')\in h_{out}}f\otimes g(i',j')\right)
\\&-\sum_{h_{out}\ni (i,j)}\left(\sum_{(i',j')\in h_{in}}f\otimes g(i',j')-\sum_{(i',j')\in h_{out}}f\otimes g(i',j')\right)
\\=&\sum_{h_{in}^1\ni i}\left(\sum_{i'\in h_{in}^1}f(i')g(j)-\sum_{i'\in h_{out}^1}f(i')g(j)\right)\\
&+\sum_{h_{in}^2\ni j}\left(\sum_{j'\in h_{in}^2}f(i)g(j')-\sum_{j'\in h_{out}^2}f(i)g(j')\right)
\\&-\sum_{h_{out}^1\ni i}\left(\sum_{i'\in h_{in}^1}f(i')g(j)-\sum_{i'\in h_{out}^1}f(i')g(j)\right)\\
&+\sum_{h_{out}^2\ni j}\left(\sum_{j'\in h_{in}^2}f(i)g(j')-\sum_{j'\in h_{out}^2}f(i)g(j')\right)
\\=&g(j)\left(\sum_{h_{in}^1\ni i}\left(\sum_{i'\in h_{in}^1}f(i')-\sum_{i''\in h_{out}^1}f(i'')\right)-\sum_{h_{out}^1\ni i}\left(\sum_{i'\in h_{in}^1}f(i')-\sum_{i''\in h_{out}^1}f(i'')\right)\right)
\\&+f(i)\left(\sum_{h_{in}^2\ni j}\left(\sum_{j'\in h_{in}^2}g(j')-\sum_{j''\in h_{out}^2}g(j'')\right)-\sum_{h_{out}^2\ni j}\left(\sum_{j'\in h_{in}^2}g(j')-\sum_{j''\in h_{out}^2}g(j'')\right)\right)
\\=&g(j)\lambda f(i)+f(i)\mu g(j)=(\lambda+\mu)(f\otimes g)(i,j).
\end{align*}
\end{small}Hence, $\lambda+\mu$ is an eigenvalue of $\Delta(\Gamma_1\,\square\,\Gamma_2)$, with eigenfunction $f\otimes g$. Since this is true for all such $f$, $\lambda$, $g$ and $\mu$, this proves the claim.
\end{proof}

\section*{Acknowledgment}
We would like to thank J\"urgen Jost for the constructive comments. We are grateful to the anonymous referees for the comments and suggestions that have greatly improved the first version of this paper.

 \bibliographystyle{elsarticle-num} 
 \bibliography{Spectral18.02.2021}

\end{document}